\documentclass[12pt,a4paper]{amsart}
\usepackage[letterpaper, margin=0.8in]{geometry}
\usepackage{amsfonts,amssymb,amsthm,epsf, graphics, bbm, mathabx, verbatim, caption, subcaption, enumerate}
\usepackage{comment}
\usepackage{hyperref}
\usepackage{diagbox}
\usepackage{amsmath}
\usepackage{amsthm}
\usepackage{amssymb}
\usepackage{cite}
\usepackage{tikz}
\usetikzlibrary{patterns}
\usepackage{float}

\newtheorem{proposition}{Proposition}
\newtheorem*{proposition1'}{Proposition 1'}

\newtheorem{lemma}{Lemma}
\theoremstyle{definition}
\newtheorem{definition}{Definition}
\newtheorem{theorem}{Theorem}
\newtheorem*{theoremA}{Theorem A}
\newtheorem*{theoremB}{Theorem B}
\newtheorem*{theoremC}{Theorem C}
\newtheorem*{theoremD}{Theorem D}

\newtheorem{example}{Example}

\theoremstyle{remark}
\newtheorem{remark}{Remark}

\newcommand{\R}{\mathbb{R}}

\newcommand{\eps}{\varepsilon}

\newcommand{\norm}[1]{\left\|#1\right\|}

\title{pinned patterns and density theorems in $\mathbb R^d$}
\author{Chenjian Wang}
\date{}

\begin{document}

\begin{abstract}
For integers $k\geq 3,d\geq 2,$ we consider the abundance property of pinned $k$-point patterns occurring in $E\subseteq \mathbb R^d$ with positive upper density $\delta(E)$.
 We show that for any fixed $k$-point pattern $V$, there is a set $E$ with positive upper density such that $E$ avoids all sufficiently large affine copies of $V$, with one vertex fixed at any point in $E$. However, we obtain a positive quantitative result, which states that for any fixed $E$ with positive upper density, there exists a $k$-point pattern $V,$ such that for any $x\in E$, the \textit{pinned scaling factor set} 
 \begin{equation*}
      D_x^V(E):=\{r> 0: \exists \text{ isometry } O \text{ such that }x+r\cdot O(V)\subseteq E\},
 \end{equation*}
 has upper density $\geq \tilde \eps>0$, where constant $\tilde \eps$ depends on $k,d$ and $\delta(E)$.
 
\end{abstract} 

\maketitle

\let\thefootnote\relax
\footnotetext{
Key Words: Pinned patterns, Pattern avoidance.

Mathematical Reviews subject classification: Primary: 11B30.}
\section{Introduction}
\subsection{Main results}
Define the \textit{upper density} of a set $A\subseteq \R^d$ as
\begin{equation*}
    \delta(A):=\limsup_{R\to \infty} \frac{\mathcal L^d(B(0,R)\bigcap A)}{R^d}
\end{equation*}
where $\mathcal L^d(\cdot)$ denotes the $d$-dimensional Lebesgue measure on $\R^d$.

In 1986, Bourgain proved the following variant of Szemer\'edi-type theorem for simplices.

\begin{theoremA}[Bourgain \cite{Bourgain_1986_SimplicesPatternFinding}]\label{Bourgain's theorem}
    Suppose set $A\subseteq \R^d$ has positive upper density
    and $V=\{p_1,p_2,...,p_d\}\subseteq \R^d$ is a set of $d$ points such that 
    \begin{equation*}
        \dim(\text{span}V)=d-1.
    \end{equation*} Then there exists a number $l$ such that for all $l'\geq l$, there is $O$ in the $d$-dimensional orthogonal group $\mathcal O(d)$ and $x\in \R^d$ such that
    \begin{equation*}
        x+l'O(V)\subseteq A.
    \end{equation*}
\end{theoremA}
The original proof of Theorem \hyperref[Bourgain's theorem]{A} is based on spherical operators \cite{Bourgain_1984_BourgainCircluarMaximalIHESPreprints}. A new proof for a stronger version of the theorem is given by Lyall and Magyar \cite{Lyall_Magyar_NewProofForBourgainTheoremOnSimplices} on the basis of  multilinear analysis. 

Inspired by Bourgain's theorem \hyperref[Bourgain's theorem]{A}, C. Wang \cite{wang_2025_PinnedPatternI} considered a pinned variant and proved a result for $2$-point patterns in all dimensions .
\begin{theoremB}\label{Wang1}
    For $d\geq 2$, there is a constant $C_d>0$ such that for all $A\subseteq \R^d$ with $\delta(A)>0$ and all $x\in A,$
    $$ \delta(D_x(A))=\limsup_{R\to \infty}\frac{\mathcal L^1(D_x(A)\bigcap (-R,R))}{R}\geq C_d \delta(A),$$
    where the \textit{pinned distance set} $D_x(A)$ is defined as 
    \begin{equation}\label{pinneddisset}
        D_x(A):=\{|x-y|:y\in A\}.
    \end{equation}
\end{theoremB}
The pinned distance set $D_x(A)$ collects all distances between a fixed point $x$ and all other points in $A$.

As we mentioned, the above result is about $k$-point patterns where $k=2$. In this case, there is only one type of patterns, which is distance. In this note, we follow this topic and study the case where $k\geq 3$. When $k\geq 3,$ the patterns become more complicated and even the analogous result of Theorem \hyperref[Wang1]{B} no longer holds. This is the following negative result.
\begin{theorem}
    \label{counterexample}
     For $k\geq 3$ and any fixed $k$-point pattern $V=\{p_1,p_2,p_3,...,p_k\}\subseteq \R^d$ that avoids 3 collinear points, there is a set $E$ with positive upper density satisfies that for any $x\in E$, there is $R(x)>0$ such that 
    \begin{equation*}
        (R(x),\infty)\bigcap D_x^{V}(E)=\emptyset,
    \end{equation*}
    where 
    \begin{equation}\label{psfs}
        D_x^V(E):=\{r> 0: \exists \text{ isometry } O\in\mathcal O(d) \text{ such that }x+rOV\subseteq E\}.
    \end{equation}
\end{theorem}
Let us give several remarks on the theorem.
\begin{itemize}
    \item Without loss of generality, we can always assume that $p_1=0$ and $p_2=(1,0,...,0)$ by translating and rescaling the pattern.
    \item Note that when $d\geq 3$, the condition here is weaker than Bourgain's original condition that the pattern $V$ satisfies $\dim($span$V)=d-1$.
    \item In the case of $k$-point pattern ($k\geq 3$), we replace the pinned distance set $D_x(E)$ with \textit{pinned scaling factor set} $D_x^V(E)$. Intuitively, the pinned scaling factor set consists of all scales at which a given pattern appears with one vertex fixed at $x$.
\end{itemize}

The proof of Theorem \ref{counterexample} is constructing an example and is provided in {Example} \ref{ThinConeExample} in Section \ref{SectionCounterExample}.

Theorem \ref{counterexample} confirms that for a fixed set with positive upper density, we cannot anticipate the abundance of pinned affine copies is true for all patterns when $k\geq 3$. However, there should be some patterns whose large pinned affine copies occur frequently since the set has positive upper density. This is the following result. 
\begin{theorem}\label{MainTheorem}
         For $d\geq 2 $, $ \eps_0>0$ and $k\geq 3$, there exist a finite set $\mathcal V=\mathcal V(d,k,\eps_0)$ of  $k$-point patterns  and a positive number $\tilde \eps(\eps_0,k,d)$ such that 
          the following holds.
        For all $A\subseteq \R^d$ with $\delta(A)\geq\eps_0$, there is
    a pattern $V\in \mathcal V$
    such that for all $x\in A$, we have
    \begin{equation*}
       \delta(D_x^V(A))\geq{\tilde \eps(\eps_0,k,d)}>0.
    \end{equation*}

\end{theorem}
\begin{remark}
    The number $\tilde \eps(\eps_0,k,d)$ can be written explicitly.
    \begin{equation*}
        {\tilde \eps(\eps_0,k,d)}=\frac{\eps(\eps_0,k,d)}{M_d}.
    \end{equation*}
    In this expression, $M_d=10^{10}\pi C_d^2$ for $k=3$ and $10C_d$ for $k\geq 4.$
\begin{equation}\label{epseps0k}
       \eps(\eps_0,k,d)=\begin{cases}
          \eps_0^2, & k=3,\\
           \exp{[-\exp((\frac{20C_d\pi}{\eps_0})^{1/c_{k-1}})]},& k\geq4.          
       \end{cases}
   \end{equation}
    $C_d$ is a constant depending on $d$ and $c_{k-1}$ comes from Szemer\'edi's theorem and depends on $k$.
\end{remark}

Let us pause to compare Theorem \ref{counterexample} and Theorem \ref{MainTheorem}. While Theorem \ref{counterexample} says one can defeat any fixed pattern by constructing the set appropriately, Theorem \ref{MainTheorem} says that one cannot defeat them all at once in any fixed set with positive upper density. No matter how one choose a dense set $A$, there will always be some pattern $V_i$ from a finite predetermined catalog that does appear frequently at every pin.

The proof is quantitative and relies on additive-combinatorial machinery, Szemer\'edi's theorem on arithmetic progressions combined with the spherical integral argument in \cite{wang_2025_PinnedPatternI}. This argument is useful when searching for some ``isosceles" pattern. We will divide the proof into two parts, one is the case $d=2$ in Section \ref{PlanarCase} and the other is the case $d\geq 3$ in {Section} \ref{SpacialCase}.

Let us use the following Figure \ref{fig:difference} to exhibit the difference between \cite{wang_2025_PinnedPatternI} and this note. 
\begin{figure}[h]
    \centering

\tikzset{every picture/.style={line width=0.75pt}} 

\begin{tikzpicture}[x=0.75pt,y=0.75pt,yscale=-1,xscale=1]

\draw    (70.07,70.61) -- (479.67,71.48) ;
\draw   (289.25,349.87) -- (310.35,328.53)(289.13,328.65) -- (310.47,349.75) ;
\draw [line width=1.5]    (70.73,100.61) -- (480.33,101.48) ;
\draw    (70.07,395.94) -- (479.67,396.81) ;

\draw (278,75) node [anchor=north west][inner sep=0.75pt]   [align=left] {\cite{wang_2025_PinnedPatternI}};
\draw (395,76) node [anchor=north west][inner sep=0.75pt]   [align=left] {This note};
\draw (100,163) node [anchor=north west][inner sep=0.75pt]   [align=left] {$\displaystyle k$: number of the \\points of the pattern };
\draw (295.33,178.4) node [anchor=north west][inner sep=0.75pt]    {$2$};
\draw (105,106) node [anchor=north west][inner sep=0.75pt]   [align=left] {$\displaystyle d$: dimension of \\the ambient space};
\draw (282,118.73) node [anchor=north west][inner sep=0.75pt]    {$d\geq 2$};
\draw (404.67,118.4) node [anchor=north west][inner sep=0.75pt]    {$d\geq \ 2$};
\draw (405,177.4) node [anchor=north west][inner sep=0.75pt]    {$k\geq \ 3$};
\draw (79,226) node [anchor=north west][inner sep=0.75pt]   [align=left] {Whether any pinned \\scaling factor set has \\psotive upper density?};
\draw (287.67,248.67) node [anchor=north west][inner sep=0.75pt]   [align=left] {Yes};
\draw (416,248) node [anchor=north west][inner sep=0.75pt]   [align=left] {No};
\draw (68,307) node [anchor=north west][inner sep=0.75pt]   [align=left] {Whether there is a pinned\\scaling factor set (for a\\specific pattern) has\\positive upper density?};
\draw (413.67,333) node [anchor=north west][inner sep=0.75pt]   [align=left] {Yes};

\end{tikzpicture}

    \caption{difference between \cite{wang_2025_PinnedPatternI} and this note}
    \label{fig:difference}
\end{figure}

Bourgain's result can be generalized in many directions. In \cite{Ziegler_2006_BourgainTHeoremInPlaneGeneralPatternDeltaNBHDVersion}, Ziegler considered all multi-point patterns on the plane and proved that all sufficiently large dilates of them can be contained in an arbitrarily small neighborhood of sets with positive upper density. Our results also give a partial answer for the pinned version of the Ziegler-type result in all dimensions. 

\subsection{Szemer\'edi's theorem} A key ingredient of our proof is the quantitative Szemer\'edi's theorem, which we record more details in Appendix \ref{Appendix:Szemeredi's theorem}. By the technique in \cite{wang_2025_PinnedPatternI} that is used to search for a type of specific ``“isosceles" patterns, we can reduce the problem to a pattern avoidance problem over the torus $\mathbb R/2\pi\mathbb Z$. Szemer\'edi's theorem gives us a nice quantitative upper bound for a subset of $\mathbb R/2\pi\mathbb Z$ that forbids certain pattern. This is Proposition \ref{MainLemma} in Section  \ref{Section:MainLemma}.

\subsection{Structure of the note}The structure of the note is as follows: In {Section} \ref{SectionCounterExample}, we will prove Proposition \ref{counterexample} with a counterexample. In {Section} \ref{SectionMainTheorem}, we will prove Theorem \ref{MainTheorem} where the whole section is divided into two cases: the planar case in Section \ref{PlanarCase} and the higher dimensional cases \ref{SpacialCase}.

\section{Proof of Theorem \ref{counterexample}}\label{SectionCounterExample}
Recall the statement of Theorem \ref{counterexample}. Its proof is based on the construction of  an example. Heuristically, for any fixed $k$-point pattern $V$ without three colinear point, we can take the smallest (positive) angle $\alpha$ formed by the points in $V$. Then one can always consider a thin cone satisfies the following conditions. 
\begin{itemize}
    \item The apex of the cone is the origin.
    \item The apex angle of the cone is smaller than the smallest angle $\alpha$ of $V$.
    \item The thin cone has positive upper density.
\end{itemize}
 If the origin is chosen as one of the vertices of the dilate of $V$, then there is no dilate of $V$ formed by the origin and other points in the thin cone, since any such shape has a smaller angle. When the scaling factor is sufficiently large, such an argument works for any fixed point in the thin cone, not only the origin. Now we make this rigorous.

\begin{example}[thin cone]\label{ThinConeExample}
   Let 
        \begin{align*}
            &\alpha:=\text{the smallest angle formed by the pattern $V$},\\
            &r_{min}:= \text{the shortest side-length of the pattern $V$.}
        \end{align*}
   By rescaling, we can assume $r_{min}=1.$ Since the pattern avoids collinear points, we have $\alpha\in (0,\pi).$
    
   Let $\alpha'=\frac{\alpha}{2^{100}d}\ll 1$. Define the solid cone $C(\alpha')$ as 
    \begin{equation*}
        C(\alpha'):=\{x\in\R^d:\angle\langle x/||x||,{e}_1\rangle\leq \alpha'/2\}.
    \end{equation*}
   The figure for spatial case is depicted in Figure \ref{figex3}. 
    \begin{figure}[h]
    \centering

\tikzset{every picture/.style={line width=0.75pt}} 

\begin{tikzpicture}[x=0.75pt,y=0.75pt,yscale=-1,xscale=1]

\draw    (56,250.32) -- (626.47,253) ;
\draw [shift={(628.47,253.01)}, rotate = 180.27] [color={rgb, 255:red, 0; green, 0; blue, 0 }  ][line width=0.75]    (10.93,-3.29) .. controls (6.95,-1.4) and (3.31,-0.3) .. (0,0) .. controls (3.31,0.3) and (6.95,1.4) .. (10.93,3.29)   ;
\draw    (56.2,300.82) -- (56.99,181.22) ;
\draw [shift={(57,179.22)}, rotate = 90.38] [color={rgb, 255:red, 0; green, 0; blue, 0 }  ][line width=0.75]    (10.93,-3.29) .. controls (6.95,-1.4) and (3.31,-0.3) .. (0,0) .. controls (3.31,0.3) and (6.95,1.4) .. (10.93,3.29)   ;
\draw    (23.95,289.75) -- (86.79,212.45) ;
\draw [shift={(88.05,210.9)}, rotate = 129.11] [color={rgb, 255:red, 0; green, 0; blue, 0 }  ][line width=0.75]    (10.93,-3.29) .. controls (6.95,-1.4) and (3.31,-0.3) .. (0,0) .. controls (3.31,0.3) and (6.95,1.4) .. (10.93,3.29)   ;
\draw [color={rgb, 255:red, 208; green, 2; blue, 27 }  ,draw opacity=1 ][line width=1.5]    (56,250.32) -- (534.29,194.42) ;
\draw    (534.29,253.22) ;
\draw [shift={(534.29,253.22)}, rotate = 0] [color={rgb, 255:red, 0; green, 0; blue, 0 }  ][fill={rgb, 255:red, 0; green, 0; blue, 0 }  ][line width=0.75]      (0, 0) circle [x radius= 3.35, y radius= 3.35]   ;
\draw    (105.8,244.02) .. controls (115.4,242.42) and (116.2,258.42) .. (108.2,255.22) ;
\draw    (110.6,243.22) .. controls (120.2,241.62) and (121,259.22) .. (113,256.02) ;
\draw [color={rgb, 255:red, 208; green, 2; blue, 27 }  ,draw opacity=1 ][line width=1.5]    (56,250.32) -- (534.29,312.02) ;
\draw  [color={rgb, 255:red, 208; green, 2; blue, 27 }  ,draw opacity=1 ][line width=1.5]  (520.12,253.22) .. controls (520.12,220.75) and (526.46,194.42) .. (534.29,194.42) .. controls (542.12,194.42) and (548.47,220.75) .. (548.47,253.22) .. controls (548.47,285.69) and (542.12,312.02) .. (534.29,312.02) .. controls (526.46,312.02) and (520.12,285.69) .. (520.12,253.22) -- cycle ;
\draw  [dash pattern={on 4.5pt off 4.5pt}]  (534.29,194.42) -- (534.29,253.22) ;

\draw (79.8,262.62) node [anchor=north west][inner sep=0.75pt]  [font=\footnotesize]  {$\alpha '=\frac{\alpha }{2^{100} d}$};
\draw (497.01,174.6) node [anchor=north west][inner sep=0.75pt]  [color={rgb, 255:red, 208; green, 2; blue, 27 }  ,opacity=1 ]  {$C( \alpha ')$};
\draw (31.2,239.8) node [anchor=north west][inner sep=0.75pt]    {$O$};
\draw (507,260.4) node [anchor=north west][inner sep=0.75pt]    {$R$};
\draw (536,219.4) node [anchor=north west][inner sep=0.75pt]    {$\sim \alpha 'R$};

\end{tikzpicture}

        \caption{Spatial case of $C(\alpha')$}
        \label{figex3}
    \end{figure}
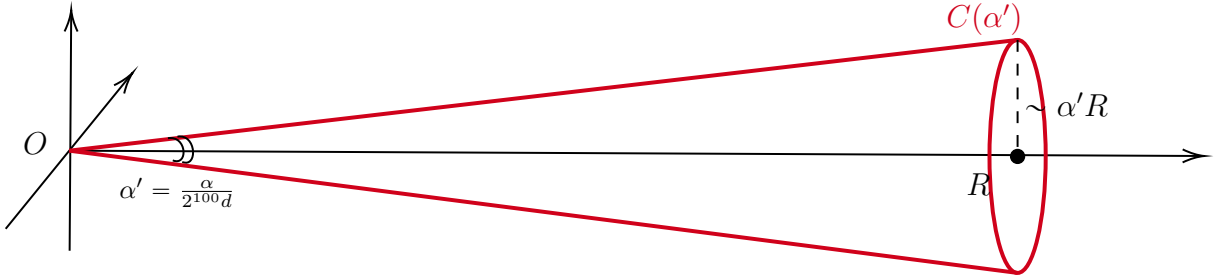

It can be checked that $\delta(C(\alpha'))>0.$ Indeed, for any $R>0$, $C(\alpha')\cap B(0,R)$ is a ``cone" with a cap from $R\mathbb S^{d-1}$ as its base. Therefore, by the volume formula of the $d-1$-dimensional cone,
\begin{equation*}
    \frac{\mathcal L^{d}(C(\alpha')\cap B(0,R))}{R^d}\geq C_d\frac{(\alpha' R)^{d-1}\cdot R}{R^d}=C_d\alpha'^{d-1}>0.
\end{equation*}

Now it suffices to prove $C(\alpha')$ does not contain any sufficiently large pinned affine copy of $V$.
We claim the following fact.
\begin{lemma}\label{xianrandedongxi}
 Fix any $x\in C(\alpha')$, there is $M(x,\alpha')>0$ such that for all $y,y'\in C(\alpha')$ such that $M(x,\alpha')\leq |y-x|\leq|y'-x|$, we have 
 \begin{equation*}
     \angle\langle y-x,y'-x\rangle\leq 2^{10}\alpha'.
 \end{equation*}
\end{lemma} 
We assume the lemma is true for the moment.
For $x\in  C(\alpha')$, let $R(x)$ be $2M(x,\alpha')$ in Lemma \ref{xianrandedongxi}. We want to show for all $R>R(x)=2M(x,\alpha'),$ there is no $R$-dilated affine copy of $V$ with $x$ as one of its vertices. Assume by contradiction that there is a $R$-dilate copy $V^R\subseteq E.$ Then the length of the shortest side is $r_{min}R=R.$ If $x$ is one of its vertex, then for any other vertices $y,y'\in V^R\setminus\{x\}$, 
\begin{equation*}
    |x-y|\geq R>R(x)=2M(x,\alpha')\quad \text{and}\quad |x-y'|\geq R>R(x)=2M(x,\alpha').
\end{equation*}
By Lemma \ref{xianrandedongxi}, 
\begin{equation*}
     \angle\langle y-x,y'-x\rangle\leq 2^{10}\alpha'=\frac{\alpha}{2^{90}d}<\alpha.
\end{equation*}
which contradicts the assumption that the smallest angle of $V$ is $\alpha$.

This concludes our construction for the counterexample.
\end{example}

Now it remains to prove Lemma \ref{xianrandedongxi}.
\begin{proof}[Proof of Lemma \ref{xianrandedongxi}]
Let $x \in C(\alpha')$ be fixed.
    We first parameterize points $y, y' \in C(\alpha')$ via spherical coordinates with ${e}_1$ as the axis:
\begin{align*}
y &= \|y\|\left(\cos\theta_y {e}_1 + \sin\theta_y {v}_y\right), \\
y' &= \|y'\|\left(\cos\theta_{y'} {e}_1 + \sin\theta_{y'} {v}_{y'}\right),
\end{align*}
where $\theta_y, \theta_{y'} \leq \alpha'/2$, ${v}_y, {v}_{y'} \perp {e}_1$, and $\|{v}_y\| = \|{v}_{y'}\| = 1$.
Similarly, decompose $x$ as:
$$
x = \|x\|\left(\cos\theta_x {e}_1 + \sin\theta_x {v}_x\right), \quad \theta_x \leq \alpha'/2.
$$
As a result,
\begin{equation}\label{aaaa}
    \begin{aligned}
y - x &= \left(\|y\|\cos\theta_y - \|x\|\cos\theta_x\right){e}_1 + \left(\|y\|\sin\theta_y {v}_y - \|x\|\sin\theta_x {v}_x\right), \\
y' - x &= \left(\|y'\|\cos\theta_{y'} - \|x\|\cos\theta_x\right){e}_1 + \left(\|y'\|\sin\theta_{y'} {v}_{y'} - \|x\|\sin\theta_x {v}_x\right).
\end{aligned}
\end{equation}
If we denote $\phi=\angle\langle y'-x,y-x\rangle$, then
\begin{equation}\label{eq:cos_angle}
\cos\phi = \frac{(y' - x) \cdot (y - x)}{\|y' - x\| \cdot \|y - x\|}.
\end{equation}
Plug equation \eqref{aaaa} to the numerator of equation \eqref{eq:cos_angle}. The coefficient of the ${e}_1$ term,
\begin{equation*}
    \begin{aligned}
        \left(\|y\|\cos\theta_y - \|x\|\cos\theta_x\right)&\left(\|y'\|\cos\theta_{y'} - \|x\|\cos\theta_x\right)
    \end{aligned}
\end{equation*}
is dominated by $\|y\|\|y'\|\cos\theta_y \cos\theta_{y'}$ when $\|y\|, \|y'\| \gg \|x\|$.
For the remaining terms,
\begin{align*}
    &\left|\left(\|y\|\sin\theta_y {v}_y - \|x\|\sin\theta_x {v}_x\right) \cdot \left(\|y'\|\sin\theta_{y'} 
    {v}_{y'} - \|x\|\sin\theta_x {v}_x\right)\right|\\
    &\leq \|y\|\|y'\|\sin\theta_y \sin\theta_{y'} + O(\|x\|(\|y\| + \|y'\|)).
\end{align*}
When $\norm{y}$ and $\norm{y'}$ are sufficiently large (depending on $x$ and $\alpha'$), the term $\|y\|\|y'\|\sin\theta_y \sin\theta_{y'}$ dominates.

Combining the terms in equation \eqref{eq:cos_angle}, we have 
$$
\cos\phi \geq \frac{\|y\|\|y'\|\cos(\theta_y + \theta_{y'}) - O(\|x\|(\|y\| + \|y'\|))}{\|y\|\|y'\|}.
$$
Since $\theta_y + \theta_{y'} \leq \alpha'$ and the continuity and monotonicity of $\cos$ function, when $\|y\|, \|y'\| \geq M(x,\alpha') \gg \|x\|$:
\begin{align*}
  \cos\phi \geq \cos\alpha' \quad \Rightarrow \quad \phi \leq \alpha'<2^{10}\alpha'.&\qedhere  
\end{align*}

\end{proof}

\section{Proof of Theorem \ref{MainTheorem}}\label{SectionMainTheorem}
We prove Theorem \ref{MainTheorem} in the following way. First, the arbitrarily chosen fixed base point $x$ can be assumed to be the origin. Therefore, it suffices to prove a ``pinned at the origin" version of Theorem \ref{MainTheorem}, which is the following Proposition \ref{MainLemma}. To prove Proposition \ref{MainLemma}, we need to use Szemer\'edi’s theorem via Gowers bounds. We reformulate the theorem to fit our setting in Lemma \ref{szmelemma}. With this lemma in hand, we split the case $d\geq 2$ in Proposition \ref{MainLemma} into two sub-cases: one is $d=2$ and the other one is $d\geq 3$. For the case of $d=2$, we reduce the problem to a pattern avoidance problem, where Lemma \ref{szmelemma} can be applied. For the case of $d=3$, we use a polar coordinate argument to lower the dimension to $2$ and repeat our reasoning in the planar case.
\subsection{Reduction to ``pinned at the origin" version}\label{Section:MainLemma}As observed in \cite{wang_2025_PinnedPatternI}, due to the definition of upper density, the main theorem is equivalent to the following ``pinned at the origin" version.
\begin{proposition}[main lemma]\label{MainLemma}
      For $d\geq 2 $, $ \eps_0>0$ and $k\geq 3$, there exist a finite set $\mathcal V=\mathcal V(d,k,\eps_0)$ of  $k$-point patterns  and a positive number $\tilde \eps(\eps_0,k,d)$ such that 
          the following holds.
        For all $A\subseteq \R^d$ with $\delta(A)\geq\eps_0$, there is
    a pattern $V\in \mathcal V$
    such that 
    \begin{equation*}
       \delta(D_0^V(A))\geq{\tilde \eps(\eps_0,k,d)}>0.
    \end{equation*}
\end{proposition}
To see this, we record the following translation invariance lemma from \cite{wang_2025_PinnedPatternI}, whose proof is direct.
\begin{lemma}[translation invariance]\label{translationinvariance}
    For all $A\subseteq \R^d$ and $x\in \R^d$, $\delta(A-x)=\delta(A).$
\end{lemma}
\begin{proof}[Proof of Theorem \ref{MainTheorem} assuming Proposition \ref{MainLemma}]
    By the conditions of Theorem \ref{MainTheorem}, $\delta(A)\geq \eps_0$ and $x\in A$. By Lemma \ref{translationinvariance}, $\delta(A-x)=\delta(A)\geq \eps_0.$ Then the set $A-x$ satisfies the condition of Proposition \ref{MainLemma}. Applying Proposition \ref{MainLemma} to $A-x$, we obtain that there is a pattern $V$ such that 
    \begin{equation*}
        \delta(D_0^V(A-x))\geq \tilde{\eps}(\eps_0,k,d)>0.
    \end{equation*}
    This concludes the proof since $D_0^V(A-x)=D_x^V(A)$.
\end{proof}

To prove Proposition \ref{MainLemma}, we adapt Szemer\'edi's theorem to our setting. 
\subsection{Adaption of Szemer\'edi's theorem} The main result of this section is the following.

\begin{lemma}[size of sets avoiding $(k-1)$-AP]\label{szmelemma}
      Fix $k\geq 3$. Let $n+1$ be a large prime number such that $n\gg k$. If $E\subseteq \R/2\pi \mathbb Z=[0,2\pi)$ satisfies for any fixed $x\in \R/2\pi\mathbb Z$ and $i\in \{1,2,...,n\}$,
   \begin{equation}\label{Nok-1AP}
       \left\{x,x+\frac{2\pi i}{n+1},x+2\cdot\frac{2\pi i}{n+1},...,x+(k-2)\cdot\frac{2\pi i}{n+1}\right\}\not\subseteq E.
   \end{equation}
        Then 
\begin{equation}\label{hehele222222}
       \mathcal L^1(E)\leq \begin{cases}
           \frac{2\pi}{n+1}, &k=3,\\
           \frac{2\pi}{(\log\log (n+1))^{c_{k-1}}
           }, & k\geq 4, 
       \end{cases}
   \end{equation}
   where $c_{k-1}$ is defined in \eqref{QuantitativeSzemeredi}.
\end{lemma}
    Let us pause to explain the lemma. 
    \begin{itemize}
        \item Since we view each number from the set of left hand side of \eqref{Nok-1AP} as an element in $\R/2\pi\mathbb Z$. Naturally, $x+j\cdot\frac{2\pi i}{n+1}= x+j\cdot\frac{2\pi i}{n+1}\bmod{2\pi}$ for each $i\in \{1,...,n\}$ and $j\in \{0,...,k-2\}$.
        \item For each fixed $k$, the set of left hand side of \eqref{Nok-1AP} is exactly a $(k-1)$-term arithmetic progression, with common difference $\frac{2\pi i}{n+1}$. As $i$ ranges over all $\{1,...,n\}$, the set rages over all possible ``$(k-1)$-term APs". Therefore, roughly speaking, the condition says that $E$ avoids all ``$(k-1)$-term APs". And the conclusion of the lemma gives a quantitative upper bound for the Lebesgue measure of such a set.
    \end{itemize}  
    In addition, note that since $n+1$ is a prime number, fix $i\in \{1,...,n\}$, $x+j\cdot \frac{2\pi i}{n+1},j=0,1,...,(k-2)$ are distinct in $\R/2\pi\mathbb Z$. 

\begin{proof}
    By change of variable, \begin{equation}\label{last???} 
        \begin{aligned}
              \mathcal L^1(E)&=\int_{\R/2\pi \mathbb Z}\chi_E(x)dx\\
              &=\int_0^{\frac{2\pi}{n+1}}\sum_{\tau=0}^{n}\chi_E(x+\frac{\tau2\pi}{n+1})dx\\
               &=\int_0^{\frac{2\pi}{n+1}}\sum_{\tau=0}^{n}\chi_{E-x}(\frac{\tau2\pi}{n+1})dx
        \end{aligned}
    \end{equation}
     Fix each $x\in (0,\frac{2\pi}{n+1})$. Since $\{\frac{\tau 2\pi}{n+1},\tau\in\{0,1,...,n\}\}\cong \mathbb Z/(n+1)\mathbb Z$. The set $E-x$ can be viewed as a subset of $\mathbb Z/(n+1)\mathbb Z$ under the canonical correspondence. Therefore, 
     \begin{equation}\label{computation1}
         \begin{aligned}
                    \mathcal L^1(E)&= \int_0^{\frac{2\pi}{n+1}}\sum_{\tau=0}^{n}\chi_{E-x}(\frac{\tau2\pi}{n+1})dx\\  
                    &=\int_{0}^{\frac{2\pi}{n+1}}\left(\int_{\mathbb Z/(n+1)\mathbb Z} \chi_{E-x}(\frac{\tau 2\pi }{n+1})d\#\tau \right)dx.
         \end{aligned}
     \end{equation}
     In the above expression, $\tau$ is viewed as an element in $\mathbb Z/(n+1)\mathbb Z$. We claim that the set 
     \begin{equation*}
        B:= \left\{\tau \in \mathbb Z/(n+1)\mathbb Z:\frac{2\pi \tau}{n+1}\in E-x\right\}
     \end{equation*}
     is a subset of $\mathbb Z/(n+1)\mathbb Z$ without $(k-1)$-A.P.. If not, by our Definition \ref{m-ap}, there are distinct 
     \begin{equation*}
         \tau_1,\tau_2,...,\tau_{k-1}\in \mathbb Z/(n+1)\mathbb Z,\quad \tau_{i+1}-\tau_i=d\mod n+1
     \end{equation*}
     such that 
     \begin{equation}\label{APinZmodn+1Z}
         \frac{2\pi\tau_j}{n+1}\in E-x,\quad \forall j=1,2,...,k-1.
     \end{equation}
    Note that in $\mathbb Z/(n+1)\mathbb Z$, 
    \begin{equation*}
        {\tau_j}=\tau_1+(j-1)d.
    \end{equation*}
     Therefore, in $\R/2\pi\mathbb Z$, 
     \begin{equation*}
         2\pi \tau_i=2\pi(\tau_1+(j-1)d)\Rightarrow \frac{2\pi \tau_j}{n+1}=\frac{2\pi(\tau_1+(j-1)d)}{n+1}.
     \end{equation*}
      Hence \eqref{APinZmodn+1Z} implies that for all $j=1,2,...,k-1,$
     \begin{equation*}
        \R/2\pi\mathbb Z\supseteq E\ni x+\frac{2\pi \tau_j}{n+1}=x+\frac{2\pi(\tau_1+(j-1)d)}{n+1}=\Big(x+\frac{2\pi\tau_1}{n+1}\Big)+\frac{(j-1)d}{n+1},
     \end{equation*}
     This means that 
     \begin{equation*}
           \left\{\Big(x+\frac{2\pi\tau_1}{n+1}\Big)+\frac{(j-1)d}{n+1}\right\}_{j=1}^{k-1}\subseteq E
     \end{equation*}
     which contradicts \eqref{Nok-1AP} with $x=x+\frac{2\pi\tau_1}{n+1},i=d$.

     Therefore, return to our computation \eqref{computation1},
   \begin{equation*}
         \begin{aligned}
                     \mathcal L^1(E)&=\int_{0}^{\frac{2\pi}{n+1}}\left(\int_{\mathbb Z/(n+1)\mathbb Z} \chi_{E-x}(\frac{\tau 2\pi }{n+1})d\#\tau \right)dx\\
                    &=\int_{0}^{\frac{2\pi}{n+1}}\left(\int_{\mathbb Z/(n+1)\mathbb Z} \chi_{B}(\tau)d\#\tau \right)dx\\
                    &=\int_{0}^{\frac{2\pi}{n+1}} \# B dx\\
         \end{aligned}
     \end{equation*}
     Apply Theorem \hyperref[SzemerediTheorem]{D} to bound $\#B$, 
     \begin{equation*}
         \begin{aligned}
             \mathcal L^1(E)& \leq \int_{0}^{\frac{2\pi}{n+1}} r_{k-1}(\mathbb Z/(n+1)\mathbb Z) dx\\
                    &\leq \begin{cases}
           \frac{2\pi}{n+1}, &k=3,\\
           \frac{2\pi}{(\log\log (n+1))^{c_{k-1}}
           }, & k\geq 4.
       \end{cases}         
         \end{aligned}
     \end{equation*}
\end{proof}

\subsection{Proof of Proposition \ref{MainLemma}}
Now we start to prove Proposition \ref{MainLemma}. We first prove it in the planar case. The higher dimensional case can be deduced with a little additional effort. In what following, we are actually proving the following result.
\begin{proposition1'}
        For $d\geq 2 $, $ \eps_0>0$ and $k\geq 3$, there is a finite set of  $k$-point patterns $\mathcal V=\mathcal V(d,k,\eps_0)$ and a constant $M_{d}$ depending on $d$ and $k$ such that the following holds. Suppose $A\subseteq \R^d$ with $\delta(A)\geq\eps_0$ and $0\in A$. Then there is
    a pattern $V\in \mathcal V$
    such that the upper density of the pinned scaling factor set satisfies 
    \begin{equation}\label{estforpin0}
       \delta(D_0^V(A))\geq\frac{\eps(\eps_0,k,d)}{M_{d}}.
    \end{equation}
\end{proposition1'}
We first record the following elementary result in linear algebra.
\begin{lemma}\label{xuanzhuan}
$d\geq 2,~m\geq 1$ are two integers. For two groups of coplanar vectors of the same length $\ell$, $U=\{u_1,...,u_m\}\subseteq \ell\mathbb S^{d-1}$ and $W=\{w_1,...,w_m\}\subseteq \ell\mathbb S^{d-1}$. Assume 
\begin{equation}
       u_i=\ell(\cos \alpha_i,\sin\alpha_i,0...,0),\quad w_i=\ell(\cos\beta_i,\sin\beta_i,0,...,0),
\end{equation}
$\alpha_i<\alpha_{i+1}, \beta_i<\beta_{i+1}$ for all $i$. If 
\begin{equation*}
    \alpha_{i+1}-\alpha_i=\beta_{i+1}-\beta_i,
\end{equation*}
then there is isometry $O\in \mathcal O(d)$ such that 
\begin{equation*}
    O(U)=W.
\end{equation*}
\end{lemma}

\begin{proof}
Swapping $U$ and $V$ if necessary, we can assume $\alpha_1\leq \beta_1$. It can be directly checked that the orthogonal matrix  
\begin{equation*}
 O'=  \begin{bmatrix}
\cos \left( \beta_1 - \alpha_1 \right) & -\sin \left( \beta_1 - \alpha_1  \right) & 0 & \cdots &0 \\
\sin \left( \beta_1 - \alpha_1  \right) & \cos \left( \beta_1 - \alpha_1  \right) & 0 & \cdots &0 \\
0& 0 & 0 & \cdots &0 \\
\vdots& \vdots & \vdots & \ddots & \vdots \\
0& 0 & 0 & \cdots &0
\end{bmatrix}
\end{equation*}
sends $U$ to $W$. 
\end{proof}

 \subsubsection{Case $d=2$}\label{PlanarCase}
   Fix $k\geq 3$. For $d=2$, assume the conclusion is false, which means there is a set $A$ with $\delta(A)\geq \eps_0$ such that for all finite sets of patterns $\mathcal V$ and all $V\in\mathcal V$, equation \eqref{estforpin0} does not hold, which is 
\begin{equation*}
    \delta(D_0^V(A))<\frac{\eps(\eps_0,k,d)}{M_{d}}
\end{equation*}
    
 Let $n=n(d,k,\eps_0)\gg1$ be a prime integer to be determined later. We choose $\mathcal V=\{V_i^k,~i=1,2,...,n\}$ as follows 

 \begin{equation*}
        V_i^k=\left\{0,e_1,\left(\cos\frac{2\pi i}{n+1},\sin\frac{2\pi i}{n+1}\right),...,\left(\cos\frac{(k-2)2\pi i}{n+1},\sin\frac{(k-2)2\pi i}{n+1}\right)\right\}\subseteq \R^d,~i=1,...,n
    \end{equation*}
    and they satisfy that  
    \begin{equation}\label{nononocontradiction}
         \delta(D^{V_i^k}_0(A))< \frac{\eps(\eps_0,k,d)}{M_{d}},~ \forall i=1,2,...,n.
    \end{equation}
    The pattern for $k=3$ and $4$ can be found in Figure \ref{figpatternin3}.
    \begin{figure}[h]
        \centering

\tikzset{every picture/.style={line width=0.75pt}} 

\begin{tikzpicture}[x=0.75pt,y=0.75pt,yscale=-1,xscale=1]

\draw   (47.18,240.41) .. controls (47.18,194.9) and (84.08,158) .. (129.59,158) .. controls (175.1,158) and (212,194.9) .. (212,240.41) .. controls (212,285.92) and (175.1,322.82) .. (129.59,322.82) .. controls (84.08,322.82) and (47.18,285.92) .. (47.18,240.41) -- cycle ;
\draw    (129.59,240.41) -- (212,240.41) ;
\draw [shift={(212,240.41)}, rotate = 0] [color={rgb, 255:red, 0; green, 0; blue, 0 }  ][fill={rgb, 255:red, 0; green, 0; blue, 0 }  ][line width=0.75]      (0, 0) circle [x radius= 2.01, y radius= 2.01]   ;
\draw [shift={(129.59,240.41)}, rotate = 0] [color={rgb, 255:red, 0; green, 0; blue, 0 }  ][fill={rgb, 255:red, 0; green, 0; blue, 0 }  ][line width=0.75]      (0, 0) circle [x radius= 2.01, y radius= 2.01]   ;
\draw    (129.59,240.41) -- (206,209.82) ;
\draw [shift={(206,209.82)}, rotate = 338.18] [color={rgb, 255:red, 0; green, 0; blue, 0 }  ][fill={rgb, 255:red, 0; green, 0; blue, 0 }  ][line width=0.75]      (0, 0) circle [x radius= 2.01, y radius= 2.01]   ;
\draw    (151,230.82) .. controls (164,229.82) and (146,245.82) .. (152,238.82) ;
\draw    (156,229.82) .. controls (169,228.82) and (151,244.82) .. (157,237.82) ;
\draw   (347.85,241.08) .. controls (347.85,195.56) and (384.74,158.67) .. (430.26,158.67) .. controls (475.77,158.67) and (512.67,195.56) .. (512.67,241.08) .. controls (512.67,286.59) and (475.77,323.49) .. (430.26,323.49) .. controls (384.74,323.49) and (347.85,286.59) .. (347.85,241.08) -- cycle ;
\draw    (430.26,241.08) -- (512.67,241.08) ;
\draw [shift={(512.67,241.08)}, rotate = 0] [color={rgb, 255:red, 0; green, 0; blue, 0 }  ][fill={rgb, 255:red, 0; green, 0; blue, 0 }  ][line width=0.75]      (0, 0) circle [x radius= 2.01, y radius= 2.01]   ;
\draw [shift={(430.26,241.08)}, rotate = 0] [color={rgb, 255:red, 0; green, 0; blue, 0 }  ][fill={rgb, 255:red, 0; green, 0; blue, 0 }  ][line width=0.75]      (0, 0) circle [x radius= 2.01, y radius= 2.01]   ;
\draw    (430.26,241.08) -- (506.67,210.49) ;
\draw [shift={(506.67,210.49)}, rotate = 338.18] [color={rgb, 255:red, 0; green, 0; blue, 0 }  ][fill={rgb, 255:red, 0; green, 0; blue, 0 }  ][line width=0.75]      (0, 0) circle [x radius= 2.01, y radius= 2.01]   ;
\draw    (451.67,231.49) .. controls (464.67,230.49) and (446.67,246.49) .. (452.67,239.49) ;
\draw    (456.67,230.49) .. controls (469.67,229.49) and (451.67,245.49) .. (457.67,238.49) ;
\draw    (430.26,241.08) -- (489.02,182.75) ;
\draw [shift={(489.02,182.75)}, rotate = 315.22] [color={rgb, 255:red, 0; green, 0; blue, 0 }  ][fill={rgb, 255:red, 0; green, 0; blue, 0 }  ][line width=0.75]      (0, 0) circle [x radius= 2.01, y radius= 2.01]   ;

\draw (140.33,253.07) node [anchor=north west][inner sep=0.75pt]  [font=\scriptsize]  {$\frac{2\pi i}{n+1}$};
\draw (237,233.4) node [anchor=north west][inner sep=0.75pt]    {$e_{1}$};
\draw (113,223.4) node [anchor=north west][inner sep=0.75pt]    {$O$};
\draw (222,181.4) node [anchor=north west][inner sep=0.75pt]  [font=\tiny]  {$\left(\cos\frac{2\pi i}{n+1} ,\sin\frac{2\pi i}{n+1} \right)$};
\draw (436.33,252.4) node [anchor=north west][inner sep=0.75pt]  [font=\scriptsize]  {$\frac{2\pi i}{n+1}$};
\draw (537.67,234.07) node [anchor=north west][inner sep=0.75pt]    {$e_{1}$};
\draw (414.67,225.07) node [anchor=north west][inner sep=0.75pt]    {$O$};
\draw (519.33,206.73) node [anchor=north west][inner sep=0.75pt]  [font=\tiny]  {$\left(\cos\frac{2\pi i}{n+1} ,\sin\frac{2\pi i}{n+1} \right)$};
\draw (115.67,335.4) node [anchor=north west][inner sep=0.75pt]    {$V_{i}^{3}$};
\draw (417,337.07) node [anchor=north west][inner sep=0.75pt]    {$V_{i}^{4}$};
\draw (500,172.07) node [anchor=north west][inner sep=0.75pt]  [font=\tiny]  {$\left(\cos\frac{2\pi i}{n+1} ,\sin\frac{2\pi i}{n+1} \right)$};

\end{tikzpicture}

        \caption{Pattern in $xOy$-plane}
        \label{figpatternin3}
    \end{figure}
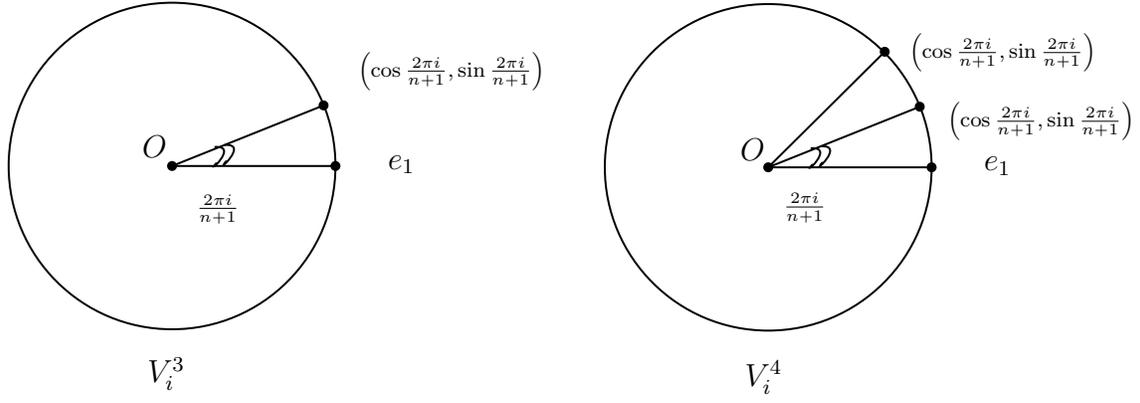


    By the definition of upper density and limit superior, equation \eqref{nononocontradiction} means that for any $\eta>0,$ there is $R(\eta,n,i)>0$ such that for all $R\geq R(\eta,n,i)$,
    \begin{equation}\label{shangshi}
        \frac{\mathcal L^1(D_0^{V_i^k}(A)\cap [0,R])}{R}<\frac{\eta}{n}+\frac{\eps(\eps_0,k,d)}{M_{d}},~\forall i=1,2,...,n,
    \end{equation}
 which is equivalent to     \begin{equation}\label{computation2}
       \mathcal L^1(D_0^{V_i^k}(A)\cap [0,R])<\frac{\eta R}{n}+\frac{\eps_0 \eps(\eps_0,k,d)R}{M_{d}},~ \forall i=1,2,...,n.
    \end{equation}

    Denote $D_0^{V_i^k}(A)$ as $D_0^{i,k}(A)$ and define the union of the pinned scaling factor sets as 
    \begin{equation*}
        D_0(n,k):=\bigcup_{i=1}^nD_0^{i,k}(A).
    \end{equation*}
     Then for all $R\geq R(\eta,n):=\max\{R(\eta,n,i):i=1,2,...,n\}$,
    \begin{equation}\label{I1est}
        \begin{aligned}
             \mathcal L^1\left( D_0(n,k)\bigcap [0,R] \right)&= \mathcal L^1\left( \bigcup_{i=1}^nD_0^{i,k}(A)\bigcap [0,R] \right)\\
             &\leq\sum_{i=1}^n\mathcal L^1(D_0^{i,k}(A)\cap [0,R])\\
             &<\eta R+\frac{n\eps(\eps_0,k,d) R}{M_{d}}.
        \end{aligned}
    \end{equation}
    In the last inequality, we apply \eqref{computation2}.
    
    On the other hand, for such $R$, we estimate the nominator of the upper density $\delta(A)$. 
    \begin{equation}\label{nominator}
         \begin{aligned}
       \mathcal L^d \left(A\bigcap B(0,R)\right)&= \int_{0}^R \sigma_r^{d-1}(r\mathbb S^{d-1}\cap A)dr\\
        &=\left(\int_{D_0(n,k)\cap [0,R]}+\int_{ [0,R]\setminus D_0(n,k)}\right)\sigma_r^{d-1}(r\mathbb S^{d-1}\cap A)dr\\
        &=I_1+I_2,
    \end{aligned}   
    \end{equation}
    where in the first line, the area measure $\sigma_r^{d-1}$ on $r\mathbb S^{d-1}$ is defined as 
    \begin{equation}\label{areameas}
       \sigma_r^{d-1}(r\mathbb S^{d-1}\cap A):= \int_{\R/2\pi \mathbb Z} \chi_{r\mathbb S^{d-1}\cap A}(r\cos \theta, r\sin\theta)d\theta r^{d-1}.
    \end{equation}

    We first address $I_1$. By \eqref{I1est}, 
    \begin{equation}\label{I1}
    \begin{aligned}
          I_1&=\int_{D_0(n,k)\cap [0,R]}\sigma_r^{d-1}(r\mathbb S^{d-1}\cap A)dr \\
          &\leq \mathcal L^1\left( D_0(n,k)\bigcap [0,R] \right)\cdot \sigma_r^{d-1}(r\mathbb S^{d-1})\\
            &< C_dR^{d-1}\cdot(\eta R+\frac{n\eps(\eps_0,k,d) R}{M_{d}})\\
          &=C_dR^d (\eta +\frac{n\eps(\eps_0,k,d) }{M_{d}}).
    \end{aligned}
    \end{equation}
    For $I_2$, 
    \begin{align*}
        I_2&=\int_{ [0,R]\setminus D_0(n,k)}\sigma_r^{d-1}(r\mathbb S^{d-1}\cap A)dr,
    \end{align*}
we can apply Lemma \ref{szmelemma}. 

Assume $r\not\in D_0(n,k)$. Define Lipschitz map
\begin{align*}
    \varphi:\R/2\pi\mathbb Z&\to \quad r\mathbb S^{d-1}\\
   \theta\quad &\mapsto r(\cos\theta,\sin\theta)
\end{align*}
This is an isomorphism. By \eqref{areameas} and change of variable, the integrand of $I_2$
    \begin{equation}\label{integrand}
        \sigma_r^{d-1}(r\mathbb S^{d-1}\cap A)=r^{d-1}\mathcal L^1(\varphi^{-1}(r\mathbb S^{d-1}\cap A)).
    \end{equation}
It suffices to estimate the right hand side. Denote
\begin{equation*}
    \varphi^{-1}(r\mathbb S^{d-1}\cap A) \overset{\triangle}{=} E_r\subseteq \R/2\pi \mathbb Z.
\end{equation*}
We claim $E_r$ satisfies the condition \eqref{Nok-1AP} of Lemma \ref{szmelemma}. If not, there is $d\in \{1,2,...,n\}$ such that distinct numbers
\begin{equation*}
       \left\{x,x+\frac{2\pi d}{n+1},x+2\cdot\frac{2\pi d}{n+1},...,x+(k-2)\cdot\frac{2\pi d}{n+1}\right\}\subseteq E_r.
   \end{equation*}  
This implies their images $S$ under $\varphi$,
\begin{equation*}
  S:=  \left\{r\left(\cos\big(x+j\cdot\frac{2\pi d}{n+1}\big),\sin\big(x+j\cdot\frac{2\pi d}{n+1}\big)\right):j=0,1,...,k-2\right\}\subseteq r\mathbb S^{d-1}\bigcap A.
\end{equation*}
It can be checked that $W=S$ and 
\begin{equation*}
   U= rV_d^k\setminus\{0\}=\left\{re_1,r\left(\cos\frac{2\pi d}{n+1},\sin\frac{2\pi d}{n+1}\right),...,r\left(\cos\frac{(k-2)2\pi d}{n+1},\sin\frac{(k-2)2\pi d}{n+1}\right)\right\}
\end{equation*}
satisfy the condition of Lemma \eqref{xuanzhuan} (switching $U,W$ if necessary). Therefore, there is $O\in \mathcal O(d)$ such that 
\begin{equation*}
    O(U)=W.
\end{equation*}
Combining this with $O(0)=0$, $0\in A$ and $\varphi(E_r)=r\mathbb S^{d-1}\cap A$, we obtain 
\begin{equation}\label{contradiction}
rO(V_d^k)=O(rV_d^k)=S\cup \{0\}\subseteq  A.
\end{equation}
By definition, this implies 
$r\in D_0^{d,k}(A)\subseteq D_0(n,k)$ which contradicts with $r\not\in D_0(n,k)$.

Return to our estimate \eqref{integrand} to the integrand of $I_2$. This means we can apply Lemma \ref{szmelemma} to $\mathcal L^1(E_r)$. Therefore, 
\begin{equation*}
    \sigma_r^{d-1}(r\mathbb S^{d-1}\cap A)=r^{d-1}\mathcal L^1(E_r)\leq R^{d-1}\cdot\begin{cases}
        \frac{2\pi}{n+1}, &k=3,\\
           \frac{2\pi}{(\log\log (n+1))^{c_{k-1}}
           }, & k\geq 4.
    \end{cases}
\end{equation*}
Plug this back to $I_2$
\begin{equation}\label{I_2}
     I_2=\int_{ [0,R]\setminus D_0(n,k)}\sigma_r^{d-1}(r\mathbb S^{d-1}\cap A)dr\leq R^d \cdot\begin{cases}
        \frac{2\pi}{n+1}, &k=3,\\
           \frac{2\pi}{(\log\log (n+1))^{c_{k-1}}
           }, & k\geq 4.
           \end{cases}
\end{equation}

Combining \eqref{I1} and \eqref{I_2}, we obtain for \eqref{nominator}
\begin{equation*}
         \mathcal L^d \left(A\bigcap B(0,R)\right)=I_1+I_2
         \leq C_dR^d (\eta +\frac{n\eps(\eps_0,k,d) }{M_{d}})+R^d \cdot \begin{cases}
        \frac{2\pi}{n+1}, &k=3,\\
           \frac{2\pi}{(\log\log (n+1))^{c_{k-1}}
           }, & k\geq 4,
           \end{cases}
\end{equation*}
which means 
\begin{equation*}
    \frac{\mathcal L^d \left(A\bigcap B(0,R)\right)}{R^d}\leq C_d\left[\eta +\frac{n\eps(\eps_0,k,d) }{M_{d}}+\begin{cases}
        \frac{2\pi}{n+1}, &k=3,\\
           \frac{2\pi}{(\log\log (n+1))^{c_{k-1}}
           }, & k\geq 4,
           \end{cases}\right].
\end{equation*}
In the computations, $C_d$ may change from line to line. The one in the definition of $\eps(\eps_0,k,d)$ is the final $C_d$.

 For $R_i\geq R(\eta,n)$ where $\{R_i\}$ is the subsequence of $R$ that attains the limit superior in $\delta(A)$,
   \begin{equation*}
   \begin{aligned}
        \eps_0&\leq  \delta(A)= \lim_{i\to\infty}\frac{\mathcal L^d\left(A\bigcap B(0,R_i)\right)}{R_i^d} \\
        &\leq C_d\left[ \eta+\frac{n\eps(\eps_0,k,d) }{M_{d}}+\begin{cases}
        \frac{2\pi}{n+1}, &k=3,\\
           \frac{2\pi}{(\log\log (n+1))^{c_{k-1}}
           }, & k\geq 4,
           \end{cases}\right]\\
       &<\frac{\eps_0}{2}<\eps_0,
   \end{aligned}
   \end{equation*}
   if we choose 
   \begin{itemize}
       \item [$k=3$]: 
       $M_{d}=10^{10}\pi C_{d}^2$, $\eta<\frac{\eps_0}{10C_{d}},$ and prime number $n+1=n(\eps_0,k,d)+1\in \big(\frac{20\pi C_d}{\eps_0}, \frac{60\pi C_d}{\eps_0}\big)$. According to Bertrand–Chebyshev theorem, such a prime exists.
       \item [$k\geq 4$]: $M_d=10C_d$, $\eta<\frac{\eps_0}{10C_d},$ and prime number $n+1$ is contained in $$ \left(\frac{1}{3}\cdot \exp\exp\Big(\big({20C_d\pi }{/\eps_0}\big)^{1/{c_{k-1}}}\Big), \exp\exp\Big(\big({20C_d\pi }{/\eps_0}\big)^{1/{c_{k-1}}}\Big)\right).$$
   \end{itemize} 
   This is a contradiction hence concludes the proof for $d=2$.
    
 \subsubsection{Case $d\geq 3$}\label{SpacialCase}
 The method can be generalized to higher dimensions by combining a polar coordinate argument. Still analyzing by contradiction, we redefine $V_i^k$ as 
 \begin{equation*}
        V_i^k=\left\{0,e_1,\left(\cos\frac{2\pi i}{n+1},\sin\frac{2\pi i}{n+1},0\right),...,\left(\cos\frac{(k-2)2\pi i}{n+1},\sin\frac{(k-2)2\pi i}{n+1},0\right)\right\}\subseteq \R^d,~i=1,...,n.
    \end{equation*}
    Proceeding to the analysis of $I_1$ and $I_2$, we need to change the argument of estimating $I_2$ since $\varphi$ is not well-defined when we are in higher dimensions. 
    
    Assume $r\not\in D_0(n,k)$. We apply the repeated polar coordinate or the change of variable formula to $r\mathbb S^{d-1}$. We parametrize the sphere by $r\omega$ where  
    \begin{equation*}
    \omega =\omega(\theta_1,...,\theta_{d-2},\phi)=
\begin{pmatrix}
\cos\theta_1 \\
\sin\theta_1 \cos\theta_2 \\
\sin\theta_1 \sin\theta_2 \cos\theta_3 \\
\vdots \\
\sin\theta_1 \cdots \sin\theta_{d-3} \cos\theta_{d-2} \\
\sin\theta_1 \cdots \sin\theta_{d-3} \sin\theta_{d-2} \cos\phi \\
\sin\theta_1 \cdots \sin\theta_{d-3} \sin\theta_{d-2} \sin\phi
\end{pmatrix}
\in \mathbb S^{d-1},
\end{equation*}
$\theta_1\in(0,2\pi),~\theta_i,\phi\in(0,\pi),~ i=2,3,...,d-2$.

One can check that the Jacobian determinant of this change of variable is 
\begin{equation*}
    r^{d-1}\prod_{j=1}^{d-2} \sin^{d-j-1} \theta_j.
\end{equation*}
This means that under this change of variable, the area $\sigma_r^{d-1}(r\mathbb S^{d-1}\cap A)$ can be written as an integral over angles $\theta_1,...,\theta_{d-2},\phi$, which is 
\begin{equation}\label{cov}
    \begin{aligned}
          \sigma_r^{d-1}(r\mathbb S^{d-1}\cap A)&=\int_{r\mathbb S^{d-1}}\chi_A(\omega')d\sigma_r^{d-1}(\omega')\\
          &=r^{d-1}\int_{\phi=0}^{\pi}\int_{\theta_{d-2}=0}^{\pi}... \int_{\theta_1=0}^{2\pi} \chi_A(r\omega(\theta_1,...\theta_{d-2},\phi)) \left(\prod_{j=1}^{d-2} \sin^{d-j-1} \theta_j\right)
d\theta_1 \cdots d\theta_{d-2} d\phi\\
&\leq r^{d-1}\int_{\phi=0}^{\pi}\int_{\theta_{d-2}=0}^{\pi}... \int_{\theta_1=0}^{2\pi} \chi_A(r\omega(\theta_1,...\theta_{d-2},\phi))
d\theta_1 \cdots d\theta_{d-2} d\phi.
        \end{aligned}
\end{equation}
Note that in equation \eqref{cov}, for any fixed $\boldsymbol{\alpha}=(\theta_2,...,\theta_{d-2},\phi),$ $\{r\omega(\theta_1,\boldsymbol{\alpha}):\theta_1\in (0,2\pi)\}$ forms a circle with radius $r$ contained in $r\mathbb S^{d-1}$. In fact, to prove this, it suffices to show that $\{\omega(\theta_1,\boldsymbol{\alpha}):\theta_1\in (0,2\pi)\}$ forms a unit circle. 

Fix all angles except $\theta_1$, that is, fix $\theta_2, \dots, \theta_{d-2}, \phi$. Define the $(d-1)$-dimensional vector:
\[
\Vec{\beta} = 
\begin{pmatrix}
\cos \theta_2 \\
\sin \theta_2 \cos \theta_3 \\
\vdots \\
\sin \theta_2 \cdots \sin \theta_{d-2} \cos \phi \\
\sin \theta_2 \cdots \sin \theta_{d-2} \sin \phi
\end{pmatrix}
\in \mathbb{R}^{d-1}.
\]
It is directly to check that $\Vec{\beta}$ is a unit vector.
Then the  $\omega(\theta_1,\boldsymbol{\alpha})$ can be rewritten as:

\[
\omega(\theta_1,\boldsymbol{\alpha}) =
\begin{pmatrix}
\cos \theta_1 \\
\sin \theta_1 \cdot \Vec{\beta}
\end{pmatrix}
=
\cos \theta_1 \cdot u + \sin \theta_1 \cdot v,
\]
where
\[
u = 
\begin{pmatrix}
1 \\
0 \\
\vdots \\
0
\end{pmatrix}, \quad
v = 
\begin{pmatrix}
0 \\
\Vec{\beta}
\end{pmatrix}
\in \mathbb{R}^d.
\]

Since $\|u\| = \|v\| = 1$ and $u \perp v$, the trajectory $\omega(\theta_1,\boldsymbol{\alpha})$ lies entirely in the 2-dimensional plane spanned by $u$ and $v$, and moves along the unit circle in that plane.

Hence, as $\theta_1$ varies, the point $\omega(\theta_1,\boldsymbol{\alpha})$ traces out the intersection of this 2-dimensional plane, which passes through the origin with the unit sphere $\mathbb{S}^{d-1}$—that is, a unit circle.

Denote the circle $\{r\omega(\theta_1,\boldsymbol{\alpha}):\theta_1\in (0,2\pi)\}$ as $\mathbf S_r(\boldsymbol{\alpha})$. It corresponds to the inner integral over $\theta_1$ 
\begin{equation}\label{InnerIntOverTHeta1}
    \int_{\theta_1=0}^{2\pi} \chi_A(r\omega(\theta_1,...\theta_{d-2},\phi))
d\theta_1.
\end{equation}

For fixed $\mathbf S_r(\boldsymbol{\alpha})$, our goal is still deducing a contradiction of form \eqref{contradiction} and apply Lemma \ref{szmelemma}. We first apply rotation $O_{\boldsymbol{\alpha}}$, such that $O_{\boldsymbol{\alpha}}(\mathbf S_r(\boldsymbol{\alpha}))=r\mathbb S^{1}\times\{0\}\subseteq \R^2\times \{0\}\subseteq \R^d$ and then repeat our planar argument. Slightly abusing the notation, we redefine 
\begin{align*}
    \varphi:\R/2\pi\mathbb Z&\to \quad r\mathbb S^{1}\times \{0\}\\
   \theta_1\quad &\mapsto r(\cos\theta_1,\sin\theta_1,0),
\end{align*}
and define
\begin{equation*}
 E_r^{\boldsymbol{\alpha}}:=\{\theta_1\in \R/2\pi\mathbb Z: \theta_1\in  \varphi^{-1}\left( O_{\boldsymbol{\alpha}}( A\cap \mathbf S_r(\boldsymbol{\alpha})) \right)\}.
\end{equation*}
Then the inner integral \eqref{InnerIntOverTHeta1} can be rewritten as 
\begin{equation}\label{inner2}
      \int_{\theta_1=0}^{2\pi} \chi_{E_r^{\boldsymbol{\alpha}}}(\theta_1))d\theta_1=\mathcal L^{1}(E_r^{\boldsymbol{\alpha}}).
\end{equation}
Similarly, we claim the set $E_r$ satisfies the condition of Lemma \ref{szmelemma}. If not, the argument is exactly the same as the planar discussion. At last, we can find an orthogonal map $O_{\boldsymbol{\alpha}}^{-1}\circ
O $ that sends certain $rV_d^k$ to $A$, which contradicts $r\not\in D_0(n,k)$. 

The claim allows us to apply Lemma \ref{szmelemma} to the integral \eqref{inner2} over $\theta_1$. Therefore, 
\begin{align*}
     \int_{\theta_1=0}^{2\pi} \chi_A(r\omega(\theta_1,...\theta_{d-2},\phi))
d\theta_1&= \int_{\theta_1=0}^{2\pi} \chi_{E_r^{\boldsymbol{\alpha}}}(\theta_1))d\theta_1=\mathcal L^1(E_r^{\boldsymbol{\alpha}})\\
&\leq \begin{cases}
           \frac{2\pi}{n+1}, &k=3,\\
           \frac{2\pi}{(\log\log (n+1))^{c_{k-1}}
           }, & k\geq 4
       \end{cases}.
\end{align*}
Combining it with \eqref{cov}, we obtain
\begin{equation*}
    \text{RHS of }\eqref{cov}\leq \pi^{d-2}R^{d-1} \cdot\begin{cases}
        \frac{2\pi}{n+1}, &k=3,\\
           \frac{2\pi}{(\log\log (n+1))^{c_{k-1}}
           }, & k\geq 4
    \end{cases} =C_dR^{d-1}\cdot\begin{cases}
        \frac{2\pi}{n+1}, &k=3,\\
           \frac{2\pi}{(\log\log (n+1))^{c_{k-1}}
           }, & k\geq 4
    \end{cases}.
\end{equation*}
Plug this back to $I_2$, 
  \begin{align*}
        I_2&=\int_{ [0,R]\setminus D_0(n,k)}\sigma_r^{d-1}(r\mathbb S^{d-1}\cap A)dr\\
        &\leq C_dR^{d}\cdot\begin{cases}
        \frac{2\pi}{n+1}, &k=3,\\
           \frac{2\pi}{(\log\log (n+1))^{c_{k-1}}
           }, & k\geq 4
    \end{cases}.
    \end{align*}
This is the higher dimensional version of our estimate for $I_2$ in \eqref{I_2}. The rest of the proof is identical to the planar case so we omit the details. 

Finally, we conclude the proof for Proposition \ref{MainLemma}.

\section{Further directions}
We discuss possible further directions in this section.

\begin{remark}
    In the proof, what we essentially work with is the scaling factors associated to the dilated pattern of $V$ where $x$ is fixed as $0\in V$. We do not know if other types of change of variable can be used to test other non-isosceles patterns.
\end{remark}

\begin{remark}
  If we denote 
  \begin{equation*}
    m:=\text{span} (V),
\end{equation*}
then the pattern we found satisfies $m=2$. One can also consider higher dimensional patterns which correspond to more complicated avoidance problems. For example, if we assume $m=3$, one possible pattern we can consider is ``equilateral triangle" on $\mathbb S^{d-1}.$ In the last integration we may leave $\theta_{1}$ and $\phi$ as variables and ask: If $E\subseteq \R/\mathbb Z\times \R/\mathbb Z$ avoids all equilateral triangle with side-length $i/n.$ What is the quantitative upper bound (depending on $n$) of $\mathcal L^2(E)$? In this case, results in \cite{Jaber_2025_AvoidTriangleInZ/NZ} may be applied.
\end{remark}

\appendix
\section{Szemer\'edi's theorem}\label{Appendix:Szemeredi's theorem}
We record the following prerequisites related to Szemer\'edi's theorem. 
\begin{definition}[$m$ term arithmetic progression, $m$-A.P.]\label{m-ap}For $m\geq 3$ and $N\gg m\geq 3,$ a sequence of $m$ elements ${a_1},{a_2},...,{a_m}\in \mathbb Z/N\mathbb Z$ is called an \textit{$m$ term arithmetic progression} (\textit{$m$-A.P.}) with common difference $d$ if 
\begin{itemize}
    \item ${a_{i+1}}-{a_i}=d\mod{N}$, where $d \in \{1,2,...,N-1\}$ for all $i$.
    \item ${a_i}\neq {a_j}$, if $i\neq j.$
\end{itemize}
For example, the common difference $d$ of 3-A.P. $6,1,3$ in $\mathbb Z/7\mathbb Z$ is $2$. We require the common difference is a number between $1,2,...,N-1$.
\end{definition}

Interestingly, such pattern existence or abundance problem can be linked to Szemer\'edi's theorem in avoidance problem. For our purpose, the following quantitative version is needed. 
\begin{theoremC}[Gowers \cite{Gowers_2001_SzemerediGowerNorm}]
    Define 
\begin{equation*}
    r_m(\mathbb Z/N\mathbb Z):=\\ \text{the cardinality of maximal subsets of $\mathbb Z/N\mathbb Z$ without $m$-A.P.}.
\end{equation*} 
Then 
\begin{equation}\label{QuantitativeSzemeredi}
     r_m(\mathbb Z/N\mathbb Z)\leq \frac{N}{(\log\log N)^{c_m}},\quad \text{ where }c_m=1/2^{2^{m+9}}. 
\end{equation}
\end{theoremC}
This result does not appear explicitly in Gowers' original paper \cite{Gowers_2001_SzemerediGowerNorm}, whereas his method with Gowers' norm works well for more general groups. One can find equation \eqref{QuantitativeSzemeredi} in Tao and Vu's book \cite[Proposition 11.12]{Tao_Vu_2006_AddCombi}. More recent results about Szemer\'edi's theorem such as \cite{Obryant_2011_SzemerediTheoremOneKase,Bloom_Sisask_2023_SemerediTheoremWithCasek=3,Green_Tao_2009_SemerediTheoremWithCasek=4II,Green_Tao_2017_SemerediTheoremWithCasek=4III,Leng_Sah_Sawhney_2024_SemerediTheoremAKase} can be applied and a tiny improvement in the quantitative bounds in Theorem \ref{MainTheorem} can be obtained. We do not do this for computational simplicity. 

If we also define \textit{$2$-A.P.} to be an ordered pair $(a,b)\in \mathbb Z/N\mathbb Z\times \mathbb Z/N\mathbb Z$, $\bar a\neq \bar b$, then trivially,  
\begin{equation*}
    r_2(\mathbb Z/N\mathbb Z)\leq 1.
\end{equation*}
Combining this with Gowers' result \eqref{QuantitativeSzemeredi}, we have 
\begin{theoremD}[Szemer\'edi's Theorem]\label{SzemerediTheorem}
\begin{equation}
    r_m(\mathbb Z/N\mathbb Z)\leq \begin{cases}
        \frac{N}{(\log\log N)^{c_m}}, & m\geq 3,\\
        1, &m=2.
    \end{cases}
\end{equation}
\end{theoremD}
We will apply Theorem \hyperref[SzemerediTheorem]{D} to prove Theorem \ref{MainTheorem}.

\bibliographystyle{plain}
\bibliography{References_of_CJW_20250903.bib}

\vspace{1em}
\noindent \textsc{Department of Mathematics, 1984 Mathematics Road, The University of British Columbia Vancouver, BC, Canada, V6T 1Z2}. \\
\textit{Email address}: \texttt{chjwang@math.ubc.ca}

\end{document}